\documentclass[11pt]{article}

\usepackage{geometry}
\usepackage{amsmath,amsfonts,graphicx,framed}
\usepackage{amssymb}
\usepackage{amsthm}
\usepackage{fancyhdr}
\usepackage{float}
\usepackage{caption}
\usepackage{comment}
\usepackage[utf8]{inputenc}
\usepackage{authblk}
\usepackage[mathscr]{euscript}
\usepackage{mathtools}
\usepackage{xcolor}

\theoremstyle{plain}
\newtheorem{thm}{Theorem}[section]
\newtheorem{lem}{Lemma}[section]
\newtheorem{cor}{Corollary}[section]

\newtheorem{prop}{Proposition}[section]
\theoremstyle{definition}

\newtheorem{rem}{Remark}[section]

\numberwithin{equation}{section}

\newcommand{\pd}{\partial}

\newcommand{\mbb}{\mathbb}

\newcommand{\ep}{\varepsilon}
\newcommand{\vp}{\varphi}
\newcommand{\re}{\mbb R}
\newcommand{\al}{\alpha}

\newcommand{\Om}{\Omega}
\newcommand{\eqal}[1]{\begin{equation}\begin{aligned}#1\end{aligned}\end{equation}}

\def\Xint#1{\mathchoice
{\XXint\displaystyle\textstyle{#1}}%
{\XXint\textstyle\scriptstyle{#1}}%
{\XXint\scriptstyle\scriptscriptstyle{#1}}%
{\XXint\scriptscriptstyle\scriptscriptstyle{#1}}%
\!\int}
\def\XXint#1#2#3{{\setbox0=\hbox{$#1{#2#3}{\int}$ }
\vcenter{\hbox{$#2#3$ }}\kern-.6\wd0}}

\def\dashint{\Xint-}

\begin{document}

\title{Hessian estimates for the sigma-2 equation in dimension four}

\author{Ravi Shankar and Yu Yuan}

\date{}

\maketitle

\begin{abstract}
We derive a priori interior Hessian estimates and interior regularity for the $\sigma_2$ equation in dimension four. Our method provides respectively a new proof for the corresponding three dimensional results and a Hessian estimate for smooth solutions satisfying a dynamic semi-convexity condition in higher $n\ge 5$ dimensions.
\end{abstract}
\maketitle

\section{Introduction}

\footnotetext[1]{\today}

In this article, we resolve the question of the interior a priori Hessian estimate
and regularity for the $\sigma_{2}$ equation
\begin{equation}
\sigma_{2}\left(  D^{2}u\right)  =\sum_{1\leq i<j\leq n}\lambda_{i}\lambda
_{j}=1 \label{s2}%
\end{equation}
in dimension $n=4,$ where $\lambda_{i}'s$ are the eigenvalues of the
Hessian $D^{2}u.$

\begin{thm}
\label{thm:s2}
Let $u$ be a smooth solution to (\ref{s2}) in the positive branch
$\bigtriangleup u>0$ on $B_{1}(0)\subset\mathbb{R}^{4}$.  Then $u$ has an
implicit Hessian estimate
\[
|D^{2}u(0)|\leq C(\left\Vert u\right\Vert _{C^{1}\left(  B_{1}\left(
0\right)  \right)  })\ \ \
\]
with$\ \left\Vert u\right\Vert _{C^{1}\left(  B_{1}\left(  0\right)  \right)
}=\left\Vert u\right\Vert _{L^{\infty}\left(  B_{1}\left(  0\right)  \right)
}+\left\Vert Du\right\Vert _{L^{\infty}\left(  B_{1}\left(  0\right)  \right)
}.$
\end{thm}

From the gradient estimate for $\sigma_{k}$-equations by Trudinger \cite{T2} and
also Chou-Wang \cite{CW} in the mid 1990s, we can bound $D^{2}u$ in terms of the
solution $u$ in $B_{2}\left(  0\right)  $ as%
\[
|D^{2}u(0)|\leq C(\left\Vert u\right\Vert _{L^{\infty}\left(  B_{2}\left(
0\right)  \right)  }).
\]

In higher $n\geq5$ dimensions, our method provides a Hessian estimate for
smooth solutions satisfying a semi-convexity type condition with movable lower
bound (\ref{dscx}), which is unconditionally valid in four dimensions by \eqref{sharp}.

\begin{thm}
\label{thm:n5}
Let $u$ be a smooth solution to (\ref{s2}) in the positive branch
$\bigtriangleup u>0$ on $B_{1}(0)\subset\mathbb{R}^{n}$ with $n\geq5,$
satisfying a dynamic semi-convex condition%
\begin{equation}
\lambda_{\min}\left(  D^{2}u\right)  \geq-c\left(  n\right)  \bigtriangleup
u\ \ \ \text{with \ \ }c\left(  n\right)  =\frac{\sqrt{3n^{2}+1}-n+1}{2n}.
\label{dscx}%
\end{equation}
Then $u$ has an implicit Hessian estimate
\[
|D^{2}u(0)|\leq C(n,\left\Vert u\right\Vert _{L^{\infty}\left(  B_{1}\left(
0\right)  \right)  }).
\]

\end{thm}

One application of the above estimates is the interior regularity
(analyticity) of $C^{0}$ viscosity solutions to (\ref{s2}) in four dimensions,
when the estimates are combined with the solvability of the Dirichlet problem
with $C^{4}$ boundary data by Caffarelli-Nirenberg-Spruck \cite{CNS} and also
Trudinger \cite{T1}. In particular, the solutions of the Dirichlet problem with
$C^{0}$ boundary data to four dimensional (\ref{s2}) of both positive branch
$\bigtriangleup u>0$ and negative branch $\bigtriangleup u<0$ respectively,
enjoy interior regularity.

Another consequence is a rigidity result for entire solutions to (\ref{s2}) of
both branches with quadratic growth, namely all such solutions must be
quadratic, provided the smooth solutions in dimension $n\geq5$ also satisfying
the dynamic semi-convex assumption (\ref{dscx}), or the symmetric one
$\lambda_{\max}\left(  D^{2}u\right)  \leq -c\left(  n\right)  \bigtriangleup
u$ in the symmetric negative branch case. Warren's rare saddle entire solution
to (\ref{s2}) shows certain convexity condition is necessary \cite{W}. Other
earlier related results can be found in \cite{BCGJ} \cite{Y1} \cite{CY} \cite{CX} \cite{SY3}.

In two dimensions, an interior Hessian bound for \eqref{s2}, the
Monge-Amp\`{e}re equation $\sigma_{2}=\det D^{2}u=1$ was found via isothermal coordinates, which are readily available under Legendre-Lewy transform, by
Heinz \cite{H} in the 1950s. The dimension three case was done via the
minimal surface structure of equation (\ref{s2}) and a full strength Jacobi inequality by Warren-Yuan in the late 2000s \cite{WY}. In higher dimensions $n\geq4$ any effective geometric structure of (\ref{s2}) appears hidden, although the level set of non-uniformly elliptic equation (\ref{s2}) is convex.

In recent years, Hessian estimates for convex smooth solutions of (\ref{s2})
have been obtained via a pointwise approach by Guan and Qiu \cite{GQ}. Hessian
estimates for almost convex smooth solutions of (\ref{s2}) have been derived
by a compactness argument in \cite{MSY}, and for semi-convex smooth solutions
in \cite{SY1} by an integral method. However, we cannot extend these a priori
estimates, including Theorem 1.2, to interior regularity statements for
viscosity solutions of (\ref{s2}), because the smooth approximations may not
preserve the convexity or semi-convexity constraints. Taking advantage of an
improved regularity property for the equation satisfied by the Legendre-Lewy
transform of almost convex viscosity solutions, interior regularity was
reached in \cite{SY2}.

For higher order $\sigma_{k}\left(  D^{2}u\right)  =1$ with $k\geq3$
equations, which is the Monge-Amp\`{e}re equation in $k$ dimensions, there are
the famous singular solutions constructed by Pogorelov \cite{P} in the 1970s, and
later generalized in \cite{U1}. Worse singular solutions have been produced in
recent years. Hessian estimates for solutions with certain strict $k$-convexity
constraints to Monge-Amp\`{e}re equations and $\sigma_{k}$ equation ($k\geq2$)
were derived by Pogorelov \cite{P} and Chou-Wang \cite{CW} respectively using the
Pogorelov's pointwise technique. Urbas \cite{U2} \cite{U3} obtained (pointwise) Hessian
estimates in term of certain integrals of the Hessian for $\sigma_{k}$ equations.  Recently, Mooney \cite{M2} derived the strict
2-convexity of convex viscosity solutions to (\ref{s2}), consequently, relying on
the solvability \cite{CNS} and a priori estimates \cite{CW}, gave a different proof of
the interior regularity of those convex viscosity solutions.

\smallskip
Our proof of Theorem \ref{thm:s2} synthesizes the ideas of Qiu \cite{Q} with Chaudhuri-Trudinger \cite{CT} and Savin \cite{S}.  Qiu showed that in dimension three, where a Jacobi inequality is valid (see Section \ref{sec:Jac} for definitions of the operators)
$$
F_{ij}\pd_{ij}\ln\Delta u\ge \ep F_{ij}(\ln\Delta u)_i(\ln\Delta u)_j,
$$
a maximum principle argument leads to a doubling, or ``three-sphere'' inequality:
$$
\sup_{B_1(0)}\Delta u\le C(n,\|u\|_{C^1(B_2(0))})\sup_{B_{1/2}(0)}\Delta u.
$$
(A lower bound condition on $\sigma_3(D^2u),$ satisfied by convex solutions of  (\ref{s2}) in general dimensions permitted Guan-Qiu to exclude the inner ``sphere" term $B_{1/2}(0)$ in the above inequality for their eventual Hessian estimates earlier in \cite{GQ}.) Iterating this ``three-sphere'' inequality shows that the Hessian is controlled by its maximum on any arbitrarily small ball.  To put it another way, any blowup point propagates to a dense subset of $B_1(0)$.  To rule out Weierstrass nowhere twice differentiable counterexamples, it suffices to find a single smooth point; Savin's small perturbation theorem \cite{S} guarantees a smooth ball if there is a smooth point.  It more than suffices to establish partial regularity, such as Alexandrov's theorem.  Chaudhuri and Trudinger \cite{CT} showed that $k$-convex functions have an Alexandrov theorem if $k>n/2$.  This gives a new compactness proof of Hessian estimate and regularity for \eqref{s2} in dimension three without minimal surface arguments, and also Hessian estimate for \eqref{s2} in general dimensions with semi-convexity assumption in \cite{SY1}, where a Jacobi inequality and Alexandrov twice differentiability are available.

\smallskip
In higher dimensions $n\ge 4$, there are three new difficulties.  Although the H\"older estimate for $k$-convex \textit{functions} may not be valid for $k\le n/2$, we can replace it with the interior gradient estimate for 2-convex \textit{solutions} in \cite{T2} \cite{CW}; this gives an Alexandrov theorem. The main hurdle is the Jacobi inequality, which fails for four and higher dimensions without a priori control on the minimum eigenvalue $\lambda_{min}$ of $D^2u$; the Jacobi inequality was discovered in \cite{SY1, SY3} for semi-convex solutions.  Instead, we can only establish an ``almost-Jacobi inequality", where $\ep\sim 1+2\lambda_{min}/\Delta u$ in four dimensions.  This choice of $\ep$ degenerates to zero for the extreme configurations $(\lambda_1,\lambda_2,\lambda_{3},\lambda_4)=(a,a,a,-a+O(1/a))$.  At first glance, $\ep\to 0$ means Qiu's maximum principle argument fails; the positive term $\ep|\nabla_F b|^2$ can no longer absorb bad terms.  On the other hand, for the extreme configurations, the equation becomes conformally uniformly elliptic.  The, usually defective, lower order term $\Delta_F|Du|^2\gtrsim \sigma_1\lambda_{min}^2$, is large enough to take control of the bad terms.  The dynamic semi-convexity assumption \eqref{dscx} allows the outlined four dimensional arguments to continue working in higher $n\ge5$ dimensions.

\smallskip
Using similar methods, a new proof of regularity for strictly convex solutions to the Monge-Amp\`ere equation is found in \cite{SY4}.  Extrinsic curvature estimates for the scalar curvature equation in dimension four are found in \cite{Sh}, extending the dimension three result of Qiu \cite{Q1}.  In forthcoming work, investigation will be done on conformal geometry's $\sigma_2$ Schouten tensor equation with negative scalar curvature and the improvement of the $W^{2,6+\delta}$ to $C^{1,1}$ estimate in \cite{D} to a $W^{2,6}$ to $C^{1,1}$ estimate.

\smallskip
In still higher dimensions $n\ge 5$, we are not even able to establish that $\ln\Delta u$ is sub-harmonic, $\ep\ge 0$, without a priori conditions on the Hessian.  There is still the problem of regularity for such solutions.  Combining the Alexandrov theorem with small perturbation [S, Theorem 1.3] only shows that the singular set is closed with Lebesgue measure zero.


\section{Almost Jacobi inequality}
\label{sec:Jac}

In \cite{SY3}, we established a Jacobi inequality for $b=\ln(\Delta u+C(n,K))$ under the semi-convexity assumption $\lambda_{min}(D^2u)\ge -K$, namely the quantitative subsolution inequality
$$
\Delta_Fb:=F_{ij}\pd_{ij}b\ge \ep\,F_{ij}b_ib_j=: \ep|\nabla_Fb|^2,
$$
where $\ep=1/3$, and for the sigma-2 equation $F(D^2u)=\sigma_2(\lambda)=1$, we denote the linearized operator by the positive definite matrix
\eqal{
\label{linearized}
(F_{ij})=\Delta u\,I-D^2u=\sqrt{2\sigma_2+|D^2u|^2}\,I-D^2u>0.
}
In dimension three, the above Jacobi inequality holds for $C(3,K)=0$ unconditionally; see \cite[$\text{p. 3207}$]{SY3} and Remark \ref{rem:3D}.  In dimension four, we can establish an inequality for $b=\ln\Delta u$ without any Hessian conditions.  The cost is that $\ep$ depends on the Hessian, and $\ep(D^2u)\to 0$ is allowed.  We obtain an ``almost" Jacobi inequality.  

\begin{prop}
\label{prop:Jac}
Let $u$ be a smooth solution to $\sigma_2(\lambda)=1$, and $b=\ln\Delta u$.  In dimension $n=4$, we have
\eqal{
\label{aJac}
\Delta_Fb\ge\ep\,|\nabla_Fb|^2,
}
where
$$
\ep=\frac{2}{9}\left(\frac{1}{2}+\frac{\lambda_{min}}{\Delta u}\right)> 0.
$$
In higher dimension $n\ge 5$, \eqref{aJac} holds for
$$
\ep=
\frac{\sqrt{3n^2+1}-(n+1)}{3(n-1)}\left(\frac{\sqrt{3n^2+1}-(n-1)}{2n}+\frac{\lambda_{min}}{\Delta u}\right)
$$
under the condition
$$
\frac{\lambda_{min}}{\Delta u}\ge -\frac{\sqrt{3n^2+1}-(n-1)}{2n}
$$
Here, $\lambda_{min}$ is the minimum eigenvalue of $D^2u$.
\end{prop}

An important ingredient for Proposition \ref{prop:Jac} is the following sharp control on the minimum eigenvalue.
\begin{lem}
\label{lem:sharp}
Let $\lambda=(\lambda_1,\dots,\lambda_n)$ solve $\sigma_2(\lambda)=1$ with $\lambda_1+\cdots+\lambda_n>0$ and $\lambda_1\ge\lambda_2\ge\cdots\ge\lambda_n$.  Then the following bound holds for $n>2$ and is sharp:
\eqal{
\label{sharp}
\sigma_1(\lambda)>\frac{n}{n-2}|\lambda_{n}|.
}
\end{lem}

\begin{proof}
The sharpness follows from the configurations
\eqal{
\label{config}
\lambda=\left(a,a,\dots,a,-\frac{(n-2)}{2}a+\frac{1}{(n-1)a}\right).
}
Next, if $\lambda_n\ge 0$, we have
$$
\sigma_1=\lambda_1+\cdots+\lambda_n\ge n\lambda_n.
$$
For $\lambda_n<0$, we write $\lambda'=(\lambda_1,\dots,\lambda_{n-1})$ and observe that $\lambda_n=(1-\sigma_2(\lambda'))/\sigma_1(\lambda')$.  We must have $\sigma_2(\lambda')>1$, as $\sigma_1(\lambda')>0$ from (\ref{linearized}), so we write
$$
\frac{\sigma_1(\lambda)}{-\lambda_n}=-1+\frac{\sigma_1(\lambda')^2}{\sigma_2(\lambda')-1}>-1+\frac{\sigma_1(\lambda')^2}{\sigma_2(\lambda')}.
$$
We write $\sigma_1(\lambda')^2$ in terms of the traceless part $\lambda'^\bot$ of $\lambda'$ and $\sigma_2(\lambda')$:
$$
\sigma_1(\lambda')^2=\frac{n-1}{n-2}(2\sigma_2(\lambda')+|\lambda'^\bot|^2).
$$
It then follows
\begin{align*}
\frac{\sigma_1(\lambda)}{-\lambda_n}&>-1+\frac{2(n-1)}{n-2}= \frac{n}{n-2}.
\end{align*}
\end{proof}

As a consequence, we obtain the following quantitative ellipticity for equation (\ref{s2}).

\begin{cor}
\label{cor:ellipticity}
Let $\lambda=(\lambda_1,\dots,\lambda_n)$ solve $f(\lambda)=\sigma_2(\lambda)=1$, with $\lambda_1+\cdots+\lambda_n>0$.  For $\lambda_1\ge \lambda_2\ge\cdots\ge \lambda_n$ and $f_i=\pd f/\pd\lambda_i$, we have
\eqal{
\label{ellipticity}
\frac{1}{\sigma_1}&\le f_1\le \left(\frac{n-1}{n}\right)\sigma_1,\\
\left(1-\frac{1}{\sqrt 2}\right)\sigma_1&\le f_i\le 2\left(\frac{n-1}{n}\right)\sigma_1,\qquad i\ge 2.
}
\end{cor}

\begin{proof}
The upper bound for $f_1=\sigma_1-\lambda_1$ comes from the easy bound $n\lambda_1\ge\sigma_1$.  The sharp upper bound for $f_n$ follows from \eqref{sharp}:
$$
f_i\le f_n=\sigma_1-\lambda_n<\left(1+\frac{n-2}{n}\right)\sigma_1.
$$
The $i=1$ lower bound goes as follows:
$$
f_1=\sigma_1-\lambda_1=\frac{2+|(0,\lambda_2,\dots,\lambda_n)|^2}{\sigma_1+\lambda_1}>\frac{2}{\sigma_1+\lambda_1}>\frac{1}{\sigma_1}.
$$
The $i\ge 2$ lower bounds for $f_i=\sigma_1-\lambda_i$ are true if $\lambda_i\le 0$.  For $\lambda_i>0$,
$$
f_i=\sigma_1-\lambda_i>\sigma_1-\sqrt{\frac{\lambda_1^2+\cdots+\lambda_i^2}{i}}>\left(1-i^{-1/2}\right)\sigma_1,
$$
where we used 
$$
\sigma_1=\sqrt{2+|\lambda|^2}>\sqrt{\lambda_1^2+\cdots+\lambda_i^2},
$$
in the last inequality.
\end{proof}
\begin{rem}
A sharp form of \eqref{ellipticity} for the $i\ge 2$ lower bounds and rougher upper bounds was first shown in [LT, (16)]. A rougher form of the lower bounds in \eqref{ellipticity}, enough for our proof of doubling Proposition \ref{prop:doub}, also follows from \cite[Lemma 3.1]{CW}, \cite[Lemma 2.1]{CY}, and \cite[(2.4)]{SY1}.
\end{rem}

\begin{proof}[Proof of Proposition \ref{prop:Jac}]
Step 1. Expression of the Jacobi inequality.  After a rotation at $x=p$, we assume that $D^2u(p)$ is diagonal.  Then $(F_{ij})=\text{diag}(f_i)$, where $f(\lambda)=\sigma_2(\lambda)$.  The following calculation was performed in \cite[$\text{p. 4}$]{SY3} for $b=\ln(\Delta u+J)$ for some constant $J$.  We repeat it below with $J=0$, for completeness.  We start with the following formulas at $x=p$:
\begin{align}
\label{gradb}
&|\nabla_Fb|^2=\sum_{i=1}^nf_i\frac{(\Delta u_i)^2}{(\Delta u)^2},\\
\label{Deltab}
    &\Delta_Fb=\sum_{i=1}^nf_i\left[\frac{\pd_{ii}\Delta u}{\Delta u}-\frac{(\pd_i\Delta u)^2}{(\Delta u)^2}\right]
\end{align}
Next, we replace the fourth order derivatives $\pd_{ii}\Delta u=\sum_{k=1}^n\pd_{ii}u_{kk}$ in \eqref{Deltab} by third derivatives.  By differentiating \eqref{s2}, we have 
\eqal{
\label{Ds2}
\Delta_FDu=(F_{ij}u_{ijk})_{k=1}^n=0.
}
Differentiating \eqref{Ds2} and using \eqref{linearized}, we obtain at $x=p$,
\begin{align*}
\sum_{i=1}^nf_i\pd_{ii}\Delta u&=\sum_{k=1}^n\Delta_Fu_{kk}=\sum_{i,j,k=1}^nF_{ij}\pd_{ij}u_{kk}=-\sum_{i,j,k=1}^n\pd_kF_{ij}\pd_{ij}u_k\\
&=\sum_{i,j,k=1}^n-(\Delta u_k\delta_{ij}-u_{kij})u_{kij}=\sum_{i,j,k=1}^nu_{ijk}^2-\sum_{k=1}^n(\Delta u_k)^2.
\end{align*}
Substituting this identity into \eqref{Deltab} and regrouping terms of the forms $u_{ \clubsuit\heartsuit\spadesuit}^2$, $u_{\clubsuit\clubsuit\heartsuit}^2,$ $u_{\heartsuit\heartsuit\heartsuit}^2$, and $(\Delta u_\clubsuit)^2$, we obtain
$$
\Delta_Fb=\frac{1}{\sigma_1}\left\{6\sum_{i<j<k}u_{ijk}^2+\left[3\sum_{i\neq j}u_{jji}^2+\sum_iu_{iii}^2-\sum_i\left((1+\frac{f_i}{\sigma_1}\right)
(\Delta u_i)^2\right]\right\}
$$
Accounting for \eqref{gradb}, we obtain the following quadratic:
$$
(\Delta_Fb-\ep|\nabla_Fb|^2)\sigma_1\ge 3\sum_{i\neq j}u_{jji}^2+\sum_{i}u_{iii}^2-\sum_i(1+\delta f_i/\sigma_1)(\Delta u_i)^2,
$$
where $\delta:=1+\ep$ here.  As in \cite{SY3}, we fix $i$ and denote $t=t_i=(u_{11i},\dots,u_{nni})$ and $e_i$ the $i$-th basis vector of $\re^n$.  Then we recall equation (2.9) from \cite{SY3} for the $i$-th term above:
$$
Q:=3|t|^2-2\langle e_i,t\rangle^2-(1+\delta f_i/\sigma_1)\langle(1,\dots,1),t\rangle^2.
$$
The objective is to show that $Q\ge 0$.  The idea in \cite{SY3} was to reduce the quadratic form to a two dimensional subspace.  In that paper, $Q\ge 0$ was shown under a semi-convexity assumption of the Hessian.  Here, we show how to remove this assumption in dimension four.  For completeness, we repeat that reduction below.

\smallskip
Step 2.  Anisotropic projection.  Equation \eqref{Ds2} at $x=p$ shows that $\langle Df,t_i\rangle=0$, so $Q$ is zero along a subspace.  We can thus replace the vectors $e_i$ and $(1,\dots,1)$ in $Q$ with their projections:
$$
Q=3|t|^2-2\langle E,t\rangle^2-(1+\delta f_i/\sigma_1)\langle L,t\rangle ^2,
$$
where
$$
E=e_i-\frac{\langle e_i,Df\rangle}{|Df|^2} Df,\qquad L=(1,\dots,1)-\frac{\langle (1,\dots,1),Df\rangle}{|Df|^2} Df.
$$
Their rotational invariants can be calculated as in [SY3, equation (2.10)]:
\eqal{
\label{invariants}
|E|^2=1-\frac{f_i^2}{|Df|^2},\qquad |L|^2=1-\frac{2(n-1)}{|Df|^2},\qquad E\cdot L=1-\frac{(n-1)\sigma_1 f_i}{|Df|^2}.
}
The quadratic is mostly isotropic: if $t$ is orthogonal to both $E$ and $L$, then $Q=3|t|^2\ge 0$, so it suffices to assume that $t$ lies in the $\{E,L\}$ subspace.  The matrix associated to the quadratic form is
$$
Q=3I-2E\otimes E-\eta L\otimes L,
$$
where $\eta=1+\delta f_i/\sigma_1=1+(1+\ep)f_i/\sigma_1$.  Since $Q$ is a quadratic form, its matrix is symmetric and has real eigenvalues.  In the non-orthogonal basis $\{E,L\}$, the eigenvector equation is
$$
\begin{pmatrix}
3-2|E|^2&-2E\cdot L\\
-\eta L\cdot E&3-\eta|L|^2
\end{pmatrix}
\begin{pmatrix}
\alpha\\
\beta
\end{pmatrix}
=
\xi
\begin{pmatrix}
\alpha\\
\beta
\end{pmatrix}.
$$
The real eigenvalues of this matrix have the explicit form
$$
\xi=\frac{1}{2}\left(tr\pm\sqrt{tr^2-4det}\right),
$$
where the trace and determinant are given by
$$
tr=6-2|E|^2-\eta|L|^2,\qquad det=9-6|E|^2-3\eta|L|^2+2\eta\left[|E|^2|L|^2-(E\cdot L)^2\right].
$$
It thus suffices to show that $tr\ge 0$ and $\det\ge 0$.  

\medskip
Step 3. Non-negativity of the trace of the quadratic form.  In [SY3], the trace was shown positive; indeed, by \eqref{invariants},
\begin{align*}
tr&=6-2\left(1-\frac{f_i^2}{|Df|^2}\right)-\left(1+\delta\frac{f_i}{\sigma_1}\right)\left(1-\frac{2(n-1)}{|Df|^2}\right)\\
&>3-\delta\frac{f_i}{\sigma_1}=\frac{(3-\delta)\sigma_1+\delta\lambda_i}{\sigma_1}\\
&\ge 3-\delta\left(1+\frac{n-2}{n}\right)\ge 0,
\end{align*}
for any 
\eqal{
\label{trace}
\delta\le \frac{3n}{2(n-1)},
}
using the bound \eqref{sharp} in the case that $\lambda_i<0$.  

\medskip
Step 4. Non-negativity of the determinant of the quadratic form.  Our new contribution here is to analyze the determinant in general.  Again by \eqref{invariants}, the determinant is
\begin{align*}
det&=\frac{6f_i^2}{|Df|^2}-\frac{3\delta f_i}{\sigma_1}+3\left(1+\frac{\delta f_i}{\sigma_1}\right)\boxed{\frac{2(n-1)}{|Df|^2}}\\
&+2\left(1+\frac{\delta f_i}{\sigma_1}\right)\left[\frac{2(n-1)\sigma_1 f_i}{|Df|^2}-\frac{n f_i^2}{|Df|^2}-\boxed{\frac{2(n-1)}{|Df|^2}}\,\,\right]\\
&>-\frac{3\delta f_i}{\sigma_1}+4\left(1+\frac{\delta f_i}{\sigma_1}\right)\frac{(n-1)\sigma_1 f_i}{|Df|^2}+\left[6-2n\left(1+\frac{\delta f_i}{\sigma_1}\right)\right]\frac{f_i^2}{|Df|^2}.
\end{align*}
Since $f_i=\sigma_1-\lambda_i$ and $\sigma_1^2=2+|\lambda|^2$, we get $|Df|^2=(n-1)\sigma_1^2-2$, so we obtain an inequality in terms of $y:=f_i/\sigma_1$:
\eqal{
\label{det}
    det\cdot\frac{|Df|^2}{\sigma_1f_i}&>\boxed{\frac{6\delta}{\sigma_1^2}}-3(n-1)\delta+4(n-1)\left(1+\delta \frac{f_i}{\sigma_1}\right)+\left[6-2n\left(1+\delta\frac{f_i}{\sigma_1}\right)\right]\frac{f_i}{\sigma_1}\\
    &>
    \left(n-1\right)\left(4-3\delta\right)+\Big[6-2n+4\left(n-1\right)\delta\Big]  y-2n\delta y^{2}\\
    &=:q_\delta(y).
}

\begin{rem} 
\label{rem:3D}
In three dimensions, the almost Jacobi inequality \eqref{aJac} becomes a
full strength one $\bigtriangleup_{F}b\geq\frac{1}{3}\left\vert \nabla
_{F}b\right\vert ^{2},\ $because in \eqref{det}, $q_{4/3}\left(  y\right)
=\frac{8}{3}\frac{f_{i}}{\sigma_{1}}\left(  1+3\lambda_{i}/\sigma_{1}\right)
>0\ $ by \eqref{sharp}.  This was observed in \cite[$\text{p. 3207}$]{SY3}.
\end{rem}

\medskip
Recalling $\delta=1+\ep$, we write $q_\delta(y)=q_1(y)+\ep\,r(y)$.  The remainder:
$$
r(y)=-3(n-1)+4(n-1)y-2ny^2=-3(n-1)+2ny\left(\frac{2(n-1)}{n}-y\right)>-3(n-1),
$$
where we used $0<y=f_i/\sigma_1 \le f_n/\sigma_1 < 2(n-1)/n$; see \eqref{ellipticity} in Corollary \ref{cor:ellipticity}.  To estimate $q_1(y)$, let us solve $0=q_1(y)=n-1+2(n+1)y-2ny^2$:
\eqal{
y_n^\pm:=\frac{n+1\pm\sqrt{1+3n^2}}{2n},\qquad y^+_n\stackrel{n=4}{=}\,\frac{3}{2}.
}
Then $q_1(y)/(y_n^+-y)=2n(y-y_n^-)$.  This linear function is minimized at the endpoint $y=0$, so if $y_n^+-y\ge 0$, we conclude
$$
q_\delta(y)\ge -2ny_n^-(y_n^+-y)-3(n-1)\ep\ge -2ny_n^-\left(y_n^+-\frac{f_n}{\sigma_1}\right)-3(n-1)\ep=0,
$$
provided
\eqal{
\ep&:= -\frac{2ny_n^-}{3(n-1)}\left(y_n^+-\frac{f_n}{\sigma_1}\right)\\
&=\frac{\sqrt{3n^2+1}-(n+1)}{3(n-1)}\left(\frac{\sqrt{3n^2+1}-(n-1)}{2n}+\frac{\lambda_n}{\sigma_1}\right)\\
&\stackrel{n=4}{=}\frac{2}{9}\left(\frac{1}{2}+\frac{\lambda_n}{\sigma_1}\right).
}
The condition $y_n^+-y=y_n^+-\frac{f_i}{\sigma_1}\ge 0$ for all $i$ is equivalent to dynamic semi-convexity,
$$
\frac{\lambda_n}{\sigma_1}\ge -\frac{\sqrt{3n^2+1}-(n-1)}{2n}.
$$
If $n=4$, all solutions satisfy this unconditionally, using \eqref{sharp}.

\medskip
Let us now check that the trace condition \eqref{trace} is also satisfied.  It suffices to have $\ep<1/2$.  Writing $\ep=c(n)(c_n+\lambda_n/\sigma_1)$, it can be shown that $c(n)$ is an increasing function of $n$ bounded by $ (\sqrt{3}-1)/3<1/4$, and $c_n$ is a decreasing function bounded by $(\sqrt{13}-1)/4<2/3$.  Combined with $\lambda_n/\sigma_1\le 1/n\le 1/2$ (see Lemma \ref{lem:sharp}), we find that $\ep<7/24$ for $n\ge 2$.

\medskip
This completes the proof of Proposition \ref{prop:Jac} in dimension $n=4$ and higher dimension $n\ge 5$.
\end{proof}


\section{The doubling inequality}

We now use the almost-Jacobi inequality in Proposition \ref{prop:Jac} to show an a priori doubling inequality for the Hessian.  

\begin{prop}
\label{prop:doub}
Let $u$ be a smooth solution of sigma-2 equation \eqref{s2} on $B_4(0)\subset \re^n$.  If $n=4$, then the following inequality is valid:
$$
\sup_{B_{2}(0)}\Delta u\le C(n)\exp\left(C(n)\|u\|_{C^1(B_3(0))}^2\right)\sup_{B_{1}(0)}\Delta u.
$$
If $n\ge 5$, the inequality is true, if we suppose also that on $B_3(0)$, there is a semi-convexity type condition
\eqal{
\label{lower}
\frac{\lambda_{min}(D^2u)}{\Delta u}\ge -c_n,\qquad c_n:=\frac{\sqrt{3n^2+1}-n+1}{2n}.
}
\end{prop}

\begin{proof}
The following test function on $B_3(0)$ is taken from \cite[Theorem 4]{GQ} and \cite[Lemma 4]{Q}:
\eqal{
\label{Pdef}
P_{\alpha\beta\gamma}:=2\ln\rho(x)+\alpha (x\cdot Du-u)+\beta|Du|^2/2+\ln \max(\bar b,\gamma^{-1}).
}
Here, $\rho(x)=3^2-|x|^2$, and $\bar b=b-\max_{B_1(0)}b$ for $b=\ln\Delta u$.  We also define $\Gamma:=4+\|u\|_{L^\infty(B_3(0))}+\|Du\|_{L^\infty(B_3(0))}$ to gauge the lower order terms, and denote by $C=C(n)$ a dimensional constant which changes line by line and will be fixed in the end.  Small dimensional positive $\gamma$, and smaller positive constants $\alpha, \beta$ depending on $\gamma$ and $\Gamma$, will be chosen later.  We also assume summation over repeated indices for simplicity of notation, where it is impossible in Section \ref{sec:Jac}.
%
%
%

\smallskip
Suppose the maximum of $P_{\alpha\beta\gamma}$ occurs at $x^*\in B_3(0)$.  If $\bar b(x^*)\le \gamma^{-1}$, then we conclude that for $C$ large enough,
\eqal{
\label{Pmax1}
\max_{B_2(0)}P_{\alpha\beta\gamma}\le C+3\alpha\Gamma+\frac{1}{2}\beta\Gamma^2+\ln\gamma^{-1}.
}
So we suppose that $\bar b(x^*)>\gamma^{-1}$.  If $|x^*|\le 1$, then again we obtain \eqref{Pmax1}, so we also assume that $1<|x^*|<3$.

\smallskip
After a rotation about $x=0$, we assume that $D^2u(x^*)$ is diagonal, $u_{ii}=\lambda_i$, with $\lambda_1\ge \lambda_2\ge\cdots\ge \lambda_n$.  At the maximum point $x^*$, we have $DP_{\alpha\beta\gamma}=0$,
\eqal{
\label{max1}
-\frac{\bar b_i}{\bar b}&=2\frac{\rho_i}{\rho}+\alpha x_ku_{ik}+\beta u_ku_{ik}\\
&=2\frac{\rho_i}{\rho}+\alpha x_i\lambda_i+\beta u_i\lambda_i,
}
and for $0\ge D^2P_{\al\beta\gamma}= (\pd_{ij}P_{\al\beta\gamma})$, we get
\eqal{
0\ge\Big(&-\frac{4\delta_{ij}}{\rho}-2\frac{\rho_i\rho_j}{\rho^2}+\alpha ( x_ku_{ijk}+ u_{ij})+\beta(u_ku_{ijk}+u_{ik}u_{jk})+\frac{\bar b_{ij}}{\bar b}-\frac{\bar b_i\bar b_j}{\bar b^2}\Big)
}
Contracting with $F_{ij}=\pd \sigma_2/\pd u_{ij}$ and using
$$
F_{ij}u_{ijk}=0,\qquad F_{ij}u_{ij}=2\sigma_2=2,\qquad F_{ij}\delta_{ij}=(n-1)\sigma_1,
$$
as well as diagonality at $x^*$, $(F_{ij})=(f_{i}\delta_{ij})$ for $f(\lambda)=\sigma_2(\lambda)$, we obtain at maximum point $x^*$,
$$
0\ge F_{ij}\pd_{ij}P_{\alpha\beta\gamma}> -4(n-1)\frac{\sigma_1}{\rho}-2\frac{f_i\rho_i^2}{\rho^2}+\beta f_i\lambda_i^2+\frac{f_i\bar b_{ii}}{\bar b}-\frac{f_i\bar b_{i}^2}{\bar b^2}.
$$
Under the assumption that $n=3,4$, or instead that $n\ge 5$ with Hessian constraint \eqref{lower}, almost-Jacobi inequality Proposition \ref{prop:Jac} is valid, and we get
\eqal{
\label{maxP}
0\ge-C\frac{\sigma_1}{\rho}-C\frac{f_i\rho_i^2}{\rho^2}+\beta f_i\lambda_i^2+\left(c_n+\frac{\lambda_n}{\Delta u}\right)\frac{f_i\bar b_{i}^2}{\bar b}-C\frac{f_i\bar b_{i}^2}{\bar b^2}.
}
Indeed, recall that $\varepsilon=c(n)(c_n+\lambda_{n}/\sigma_1)$ in Proposition \ref{prop:Jac}.  If we divide the inequality by $c(n)$ and enlarge $C$, we obtain \eqref{maxP}.  

If the nonnegative coefficient of $f_i\bar b_i^2/\bar b$ is positive, we can proceed as in Qiu's proof. 
In the alternative case, we must use the $\beta$ term.  
We start with the latter case.  Note that from \eqref{sharp} in Lemma 1, condition \eqref{lower} $\lambda_n /\Delta u > -1/2 = - c_n$ is automatically satisfied for  $n=4$, and $\lambda_n /\Delta u > -1/3  > -c_n/2$ for  $n=3$.

\medskip
\textbf{CASE} $-c_n\le \lambda_n/\Delta u\le -c_n/2$:  It follows from \eqref{ellipticity} that $f_n\lambda_n^2\ge c(n)\sigma_1^3$. For larger $C$,
$$
0\ge -C\frac{\sigma_1}{\rho^2}+\beta \sigma_1^3-Cf_i\frac{\bar b_i^2}{\bar b^2}.
$$
Using \eqref{max1}, ellipticity \eqref{ellipticity}, and the estimate $|\lambda|^2=\sigma_1^2-2\sigma_2<\sigma_1^2$, we obtain
$$
\beta \sigma_1^3\le C\frac{\sigma_1}{\rho^2}+C(\alpha^2+\beta^2\Gamma^2)\sigma_1^3.
$$
If the small parameters satisfy
\eqal{
\label{small1}
\alpha^2\le \beta/(3C),\qquad \beta\le 1/(3C\Gamma^2),
}
we obtain $\rho^2\sigma_1^2\le C/\beta$.  Since $\sigma_1=\sqrt{2+|\lambda|^2}>\sqrt 2$, we have $\sigma_1^2>2\ln\sigma_1$, and we conclude from \eqref{Pdef} and \eqref{small1} that
\eqal{
\label{Pmax2}
P_{\alpha\beta\gamma}\le C+\ln\beta^{-1}.
}

\medskip
We next show that Qiu's argument goes through, in the case that ``almost" Jacobi becomes a regular Jacobi.

\medskip
\textbf{CASE $\lambda_n/\Delta u\ge -c_n/2$.}  It follows that, after enlarging $C$, \eqref{maxP} can be reduced to
$$
0\ge -C\frac{\sigma_1}{\rho}-C\frac{f_i\rho_i^2}{\rho^2}+\beta f_i\lambda_i^2+(\bar b-C)f_{i}\frac{\bar b_i^2}{\bar b^2}.
$$
Using $\bar b(x^*)\ge \frac{1}{2}\bar b(x^*)+\frac{1}{2}\gamma^{-1}$, we assume that $\gamma$ satisfies
\eqal{
\label{small1a}
\frac{1}{2}\gamma^{-1}\ge C,
}
so after enlarging $C$ again, we can further reduce it to
\eqal{
\label{maxP1}
0\ge -C\frac{\sigma_1}{\rho}-C\frac{f_i\rho_i^2}{\rho^2}+\beta f_i\lambda_i^2+\bar bf_{i}\frac{\bar b_i^2}{\bar b^2}.
}

\smallskip
\textbf{SUBCASE} $1<|x^*|<3$ and $x_1^2 > 1/n$: 
If the small parameters satisfy the condition
\eqal{
\label{small2}
\beta\le \alpha/(2n\Gamma),\qquad 
}
we then obtain from \eqref{max1},
$$
\frac{\bar b_1^2}{\bar b^2}\ge \frac{1}{2}(\alpha/n-\beta\Gamma)^2\lambda_1^2-\frac{C}{\rho^2}\ge\frac{1}{8n^2}\alpha^2\lambda_1^2-\frac{C}{\rho^2}.
$$
We assume that this gives a lower bound, or that $C/\rho^2\le \alpha^2\lambda_1^2/(16n^2)$:
\eqal{
\label{B1}
\frac{\bar b_1^2}{\bar b^2}\ge \frac{1}{16n^2}\alpha^2\lambda_1^2.
}
For if not, we get $\rho^2\lambda_1^2\le C/\alpha^2$.  Since $\lambda_{1}\ge\sigma_1/n$, we can get $\rho^2\ln\sigma_1\le C/\alpha^2$.  Using \eqref{Pdef} and \eqref{small1}, we would obtain
\eqal{
\label{Pmax3}
P_{\alpha\beta\gamma}\le C+2\ln\alpha^{-1}.
}
It follows then, from \eqref{B1} and \eqref{ellipticity}, that \eqref{maxP1} can be simplified to
$$
0\ge -C\frac{\sigma_1}{\rho^2}+\bar b\,f_1(\alpha^2 \lambda_1^2).
$$
From \eqref{ellipticity}, there holds $f_1\lambda_1^2\ge \sigma_1/n^2$, so we conclude $\rho^2\bar b\le C/\alpha^2$.  By \eqref{Pdef} and \eqref{small1}, we conclude a similar bound \eqref{Pmax3}:
$$
P_{\alpha\beta\gamma}\le C+2\ln\alpha^{-1}.
$$

\smallskip
\textbf{SUBCASE} $1<|x^*|<3$ and $x_k^2>1/n$ for some $k\ge 2$: Let us first note that $\sigma_1/\rho\le C f_k\rho_k^2/\rho^2$, by \eqref{ellipticity}.  We apply $\bar b>\gamma^{-1}$ to \eqref{maxP1}:
$$
0\ge -C\frac{f_i\rho_i^2}{\rho^2}+\beta f_i\lambda_i^2+\gamma^{-1}f_{i}\frac{\bar b_i^2}{\bar b^2}.
$$
Using the $DP=0$ equation, \eqref{max1} and enlarging $C$, we obtain
\eqal{
\label{maxP2}
0&\ge -C\frac{f_i\rho_i^2}{\rho^2}+\beta f_i\lambda_i^2+\gamma^{-1}f_i\frac{\rho_i^2}{\rho^2}-C\gamma^{-1}\alpha^2f_ix_i^2\lambda_i^2-C\gamma^{-1}\Gamma^2\beta^2 f_i\lambda_i^2\\
&\ge \frac{f_i\rho_i^2}{\rho^2}(\gamma^{-1}-C)+\Gamma^{-2}f_i\lambda_i^2\Big((\Gamma^2\beta)-C\gamma^{-1}(\Gamma\alpha)^2-C\gamma^{-1}(\Gamma^2\beta)^2\Big).
}
The first term is handled if $\gamma^{-1}$ is large enough:
$$
\gamma^{-1}\ge 2C.
$$
We choose $\alpha,\beta$ as follows:
\eqal{
\label{small3}
\alpha=\gamma^4/\Gamma,\qquad \beta=\gamma^{6}/\Gamma^2.
}
Let us check that the previous $\alpha,\beta$ conditions \eqref{small1} and \eqref{small2} are satisfied for any $\gamma^{-1}\ge 2C$, if $C$ is large enough:
$$
\frac{\alpha^2}{\beta}=\gamma^{2}\le \frac{1}{4C^2}<\frac{1}{3C},\qquad \frac{\Gamma\beta}{\alpha}=\gamma^2\le \frac{1}{4C^2}<\frac{1}{2n}.
$$
Finally, the coefficient of $\Gamma^{-2}f_i\lambda_i^2$ in \eqref{maxP2} is
$$
\gamma^6-C\gamma^7-C\gamma^{11}=\gamma^6(1-C\gamma-C\gamma^5)\ge \gamma^6\left(1-\frac{1}{2}-\frac{\gamma^4}{2}\right)>0.
$$
Overall, we obtain a contradiction to \eqref{maxP2}.  

\medskip
We conclude that for all large $\gamma^{-1}\ge 2C$ and $\alpha,\beta$ satisfying \eqref{small3}, the maximum of $P_{\alpha\beta\gamma}$ obeys the largest of the $P$ bounds \eqref{Pmax1}, \eqref{Pmax2}, and \eqref{Pmax3}:
$$
\max_{B_2(0)}P_{\alpha\beta\gamma}\le C+\ln\max(\gamma^{-1},\beta^{-1},\alpha^{-2})=C+\ln(\Gamma^2\gamma^{-8}).
$$
We now choose large $\gamma^{-1}=2C=C(n)$.  By \eqref{Pdef}, we obtain the doubling estimate
$$
\frac{\max_{B_2(0)}\sigma_1}{\max_{B_1(0)}\sigma_1}\le \exp\exp\left(C+\ln\Gamma^2\right)=\exp (C\Gamma^2).
$$
\end{proof}

We now modify the doubling inequality to account for ``moving centers".  We may control the global maximum by the maximum on any small ball.
\begin{cor}
\label{cor:doub}
Let $u$ be a smooth solution of sigma-2 equation \eqref{s2} on $B_4(0)\subset\re^n$.  If $n=4$, or if lower bound \eqref{lower} holds for $n\ge 5$, then the following inequality is true for any $y\in B_{1/3}(0)$ and $0<r<4/3$:
\eqal{
\label{doubley}
\sup_{B_{2}(0)}\Delta u\le C(n,r,\|u\|_{C^1(B_3(0))})\sup_{B_{r}(y)}\Delta u.
}
\end{cor}
\begin{proof}
We first note that
$$
B_1(0)\subset B_{4/3}(y)\subset B_{5/3}(y)\subset B_2(0),
$$
for any $|y|<1/3$.  By Proposition \ref{prop:doub}, we find an inequality independent of the center:
\eqal{
\label{doublex}
\sup_{B_{5/3}(y)}\Delta u\le C(n)\exp\Big(C(n)\|u\|^2_{C^1(B_3(0))}\Big)\sup_{B_{4/3}(y)}\Delta u.
}
We iterate this inequality about $y$ using the rescalings
$$
u_{k+1}(\bar x)=\left(\frac{5}{4}\right)^2u_k\left(\frac{4}{5}(\bar x-y)+y\right),\qquad u_0=u,\qquad k=0,1,2,\dots
$$ 
It follows that each $u_k$ satisfies \eqref{doublex}.  Denoting 

$$
C_k=C(n)\exp \left[C(n)\|u_k\|^2_{C^1(B_3(0))}\right]\le C(n)\exp\left[ \left(\frac{5}{4}\right)^{2k}C(n)\|u\|_{C^1(B_3(0))}\right],
$$
we obtain for $k=1,2,\dots$,
$$
\sup_{B_{5/3}(y)}\Delta u\le C_0C_1\cdots C_{k}\sup_{B_{r_{k+1}}(y)}\Delta u\le C(k,n,\|u\|_{C^1(B_3(0))})\sup_{B_{r_{k+1}}(y)}\Delta u,\qquad r_k=\frac{5}{3}\left(\frac{4}{5}\right)^{k}
$$
Letting $r_{k+1}\le r< r_k$ for some $k$, we combine this inequality with Proposition \ref{prop:doub} again, to arrive at \eqref{doubley}.
\end{proof}

\begin{rem}
In the uniformly elliptic case, or $a^{ij}b_{ij}\ge a^{ij}b_ib_j$ for $\lambda I\le (a^{ij})\le \Lambda I$, it follows from Trudinger \cite[$\text{p. 70}$]{T3} that a local Alexandrov maximum principle argument gives an integral doubling inequality:
$$
\sup_{B_1(0)}b\le C\left(n,r,\frac{\Lambda}{\lambda}\right)\left(1+\|b\|_{L^n(B_r(0)}\right).
$$
In the $\sigma_2$ case, we can find an integral doubling inequality by modifying Qiu's argument, but the non-uniform ellipticity adds a nonlinear weight to the integral:
$$
\sup_{B_1(0)}\ln\Delta u\le C(n,r)\Gamma^2\left(1+\|(\Delta u)^{2/n}\ln\Delta u\|_{L^n(B_{r}(0)}\right).
$$
This \textit{nonlinear} doubling inequality can be employed to reach Theorems \ref{thm:s2} and \ref{thm:n5}, 
as in Section \ref{sec:proof}, Step 3.
\end{rem}

\section{Alexandrov regularity for viscosity solutions}

We modify the approach of Evans-Gariepy \cite{EG} and Chaudhuri-Trudinger \cite{CT} to show the following Alexandrov regularity.  In \cite[Theorem 1, section 6.4]{EG}, the Alexandrov theorem is seen to arise from combining a gradient estimate with a ``$W^{2,1}$ estimate" for convex functions.  The latter can be heuristically understood from the a priori divergence structure calculation
$$
\int_{B_1(0)}|D^2u|\,dx \leq \int_{B_1(0)} \Delta u\,dx\le C(n)\|u\|_{L^\infty(B_2(0))}.
$$  
However, for $k$-convex functions, there is no gradient estimate, in general, and only H\"older and $W^{1,n+}$ estimates for $k>n/2$.  We are not able to use Chaudhuri and Trudinger's result in dimension $n=4$.  Yet, 2-convex solutions of $\sigma_2=1$ have an even stronger interior Lipschitz estimate, by Trudinger \cite{T2}, and also Chou-Wang \cite{CW}, with a similar ``$W^{2,1}$ estimate" from $\Delta u=\sqrt{2+|D^2u|^2}$, so the method of \cite{EG} and \cite{CT} can be applied verbatim.  We record the modifications below, for completeness.

\begin{prop}
\label{prop:Alex}
Let $u$ be a viscosity solution of sigma-2 equation \eqref{s2} on $B_4(0)$ with $\Delta u>0$.  Then $u$ is twice differentiable almost everywhere in $B_4(0)$, or for almost every $x\in B_4(0)$, there is a quadratic polynomial $Q$ such that
$$
\sup_{y\in B_{r}(x)}|u(y)-Q(y)|=o(r^2).
$$
\end{prop}

We begin the proof of this proposition by first recalling the weighted norm Lipschitz estimate \cite[Corollary 3.4, $\text{p. 587}$]{TW} for smooth solutions of $\sigma_2=1,\Delta u>0$ on a smooth, strongly convex 
domain $\Omega\subset\re^n$:
\eqal{
\label{gradp}
\sup_{x,y\in\Omega:x\neq y}d_{x,y}^{n+1}\frac{|u(x)-u(y)|}{|x-y|}\le C(n)\int_\Om |u|dx,
}
where $d_{x,y}=\min(d_x,d_y)$, and $d_x=\text{dist}(x,\pd\Omega)$.  By solving the Dirichlet problem \cite{CNS} with smooth approximating boundary data, this pointwise estimate holds for viscosity solution $u$, if $\Omega\subset\subset B_4(0)$, i.e. $u$ is locally Lipschitz.  By Rademacher's theorem, $u$ is differentiable almost everywhere, with $Du\in L^\infty_{loc}$ equal the weak (distribution) gradient.  By Lebesgue differentiation, for almost every $x\in B_4(0)$,
\eqal{
\label{gradi}
&\lim_{r\to 0}\dashint_{B_r(x)}|Du(y)-Du(x)|dy=0.
}
For second order derivatives, we next recall the definition \cite[$\text{p. 306}$]{CT} that a continuous 2-convex function satisfies both $\Delta u>0$ and $\sigma_2>0$ in the viscosity sense.  Since viscosity solution $u$ to $\sigma_2=1$ and $\Delta u>0$ is 2-convex, we deduce from \cite[Theorem 2.4]{CT} that the weak Hessian $\pd^2u$, interpreted as a vector-valued distribution, gives a vector-valued Radon measure $[D^2u]=[\mu^{ij}]$: 
$$
\int u\,\vp_{ij}\, dx=\int \vp \,d\mu^{ij},\qquad\vp\in C^\infty_0(B_4(0)).
$$

Let us outline another proof. Noting that $\sum_{j\neq i}D_{jj}u\ge \bigtriangleup u-\lambda_{\max}\geq0$ for 2-convex smooth function $u,$ where the last inequality follows from \eqref{linearized} with $2\sigma_{2}\geq0,$ via smooth approximation in $C^{0}/L^{\infty}$ norm, we see that $\mu^{i}$ and also $\mu^{e}$ for any unit vector on $\mathbb{R}^{n}$ in \cite[(2.7)]{CT}, are non-negative Borel measures, in turn, bounded on compact sets, that is,  Radon measures for 2-convex continuous $u.$ Readily $\mu^{I_{n}}$ for 2-convex continuous $u$ in \cite[(2.6)]{CT} is also a Radon measure. Consequently, for 2-convex continuous $u,$ $D_{ii}u=\mu^{I_{n}}-\mu^{i}$ and also $D_{ee}u=\mu^{I_{n}}-\mu^{e}$ in \cite[(2.8)]{CT} are Radon
measures. This leads to another way in showing that the Hessian measures $D_{ij}u=\mu^{ij}=\left(\mu^{e_{+}e_{+}}-\mu^{e_{-}e_{-}}\right)  /2$ with $e_{+}=\left(  \partial_{i}+\partial_{j}\right)  /\sqrt{2}$ and $e_{-}=\left(\partial_{i}-\partial_{j}\right)  /\sqrt{2}$  in \cite[(2.9)]{CT}, are Radon measures for all $1\leq i,j\leq n$ and 2-convex continuous $u.$

\smallskip
By Lebesgue decomposition, we write $[D^2u]=D^2u\,dx+[D^2u]_s$, where $D^2u\in L^1_{loc}$ denotes the absolutely continuous part with respect to $dx$, and $[D^2u]_s$ is the singular part.  In particular, for $dx$-almost every $x$ in $B_4(0)$,
\begin{align}
\label{Hess1}
&\lim_{r\to 0}\dashint_{B_r(x)}|D^2u(y)-D^2u(x)|dy=0,\\
\label{Hess2}
&\lim_{r\to 0}\frac{1}{r^n}\|[D^2u]_s\|(B_r(x))=0.
\end{align}
Here, we denote by $\|[D^2u]_s\|$ the total variation measure of $[D^2u]_s$.  In fact, these conditions plus \eqref{gradi} are precisely conditions (a)-(c) in \cite[Theorem 1, section 6.4]{EG}.  We state their conclusion as a lemma, and include their proof of this fact in the Appendix.

\begin{lem}
\label{lem:o(r^2)}
Let $u\in C(B_4(0))$ have a weak gradient $Du\in L^1_{loc}$ which satisfies \eqref{gradi} for a.e. $x$, and a weak Hessian $\pd ^2u$ which induces a Radon measure $[D^2u]=D^2u\,dx+[D^2u]_s$ obeying conditions \eqref{Hess1} and \eqref{Hess2} for a.e. $x$.  Then for a.e. $x\in B_4(0)$, it follows that
\eqal{
\dashint_{B_r(x)}\Big |u(y)-u(x)-(y-x)\cdot Du(x)-\frac{1}{2}(y-x)D^2u(x)(y-x)\Big|dy=o(r^2).
}
\end{lem}

Choose $x$ for which conditions \eqref{gradi}, \eqref{Hess1}, and \eqref{Hess2} are valid.  Let $h(y)=u(y)-u(x)-(y-x)\cdot Du(0)-(y-x)\cdot D^2u(0)\cdot (y-x)/2$.  Using
\eqal{
\label{L^1}
\dashint_{B_r(x)}|h(y)|dy=o(r^2),
}
we will upgrade this to the desired $\|h\|_{L^\infty(B_{r/2}(x))}=o(r^2)$.  The crucial ingredient is a pointwise estimate: for $0<2r<4-|x|$,
\eqal{
\label{point}
\sup_{y,z\in B_r(x),y\neq z}\frac{|h(y)-h(z)|}{|y-z|}\le \frac{C(n)}{r}\dashint_{B_{2r}(x)}|h(y)|dy+Cr,
}
where $C=C(n)|D^2u(x)|$.  This was shown as \cite[Lemma 3.1]{CT} for $k$-convex functions with $k>n/2$ using the H\"older estimate \cite[Theorem 2.7]{TW}, and \cite[Claim \#1, $\text{p. 244}$]{EG} for convex functions using a gradient estimate, respectively.  

\begin{proof}[Proof of \eqref{point}]
To establish \eqref{point}, we first let $g(y)=u(y)-u(x)-(y-x)\cdot Du(x)$; then $\sigma_2(D^2g(y))=1$ with $\Delta g(y)>0$, so gradient estimate \eqref{gradp} yields
\begin{align}
\label{pointLem}
\nonumber
    r^{n+1}\sup_{y,z\in B_r(x),y\neq z}&\frac{|g(y)-g(z)|}{|y-z|}\\
\nonumber
    &=\text{dist}(\boxed{\pd B_r(x)},\pd B_{2r}(x))^{n+1}\sup_{y,z\in \boxed{B_r(x)},y\neq z}\frac{|g(y)-g(z)|}{|y-z|}\\
\nonumber
    &\le \sup_{y,z\in B_{2r}(x),y\neq z}d_{y,z}^{n+1}\frac{|g(y)-g(z)|}{|y-z|}\\
\nonumber
    &\stackrel{\eqref{gradp}}{\le}C(n)\int_{B_{2r}(x)}|g(y)|dy\\
    &\le C(n)\int_{B_{2r}(x)}|h(y)|dy+C(n)|D^2u(x)|\,r^{n+2},
\end{align}
where $d_{y,z}:=\min(2r-|y-x|,2r-|z-x|)$.  Next, we polarize
$$
(y-x)\cdot D^2u(x)\cdot(y-x)-(z-x)\cdot D^2u(x)\cdot(z-x)=(y-x+z-x)\cdot D^2u(x)\cdot(y-z),
$$
which gives
$$
r^{n+1}\sup_{y,z\in B_r(x),y\neq z}\frac{|h(y)-h(z)|}{|y-z|}\le r^{n+1}\sup_{y,z\in B_r(x),y\neq z}\frac{|g(y)-g(z)|}{|y-z|}+C(n)r^{n+2}|D^2u(x)|.
$$
This inequality and \eqref{pointLem} lead to \eqref{point}.  
\end{proof}

The rest of the proof follows \cite[Claim \#2, $\text{p. 244}$]{EG} or \cite[Proof of Theorem 1.1, $\text{p. 311}$]{CT} verbatim.  We summarize the conclusion as a lemma and include its proof in the appendix.

\begin{lem}
\label{lem:pto(r^2)}
Let $h(y)\in C(B_4(0))$ and $x\in B_4(0)$ satisfy integral \eqref{L^1} and pointwise \eqref{point} bounds for $0<2r<4-|x|$.  Then $\sup_{B_{r/2}(x)}|h(y)|=o(r^2)$.  
\end{lem}

This completes the proof of Proposition \ref{prop:Alex}.

\begin{rem} In fact, Proposition \ref{prop:Alex} holds true for (continuous)
viscosity solutions to $\sigma_{k}\left(  D^{2}u\right)  =1$ for $2\leq k\leq
n/2$ in $n$ dimensions, because the needed conditions \eqref{gradp}-\eqref{Hess2} in the proof
are all available. The twice differentiability a.e. for all $k$-convex functions
and $k>n/2,$ without satisfying any equation in $n$ dimensions,  is the content
of the theorems by Alexandrov \cite[$\text{p. 242}$]{EG} and Chaudhuri-Trudinger \cite{CT}.
\end{rem}

\section{Proof of Theorems \ref{thm:s2} and \ref{thm:n5}}
\label{sec:proof}

Step 1. After scaling $4^{2}u(x/4)$, we claim that the Hessian $D^2u(0)$ is controlled by $\|u\|_{C^1(B_4(0))}$.  Otherwise, there exists a sequence of smooth solutions $u_k$ of \eqref{s2} on $B_4(0)$ with bound $\|u_k\|_{C^1(B_{3}(0))}\le A$, but $|D^2u_k(0)|\to\infty$, in either dimension $n=4$, or in higher dimension $n\ge 5$ with dynamic semi-convexity \eqref{lower}.  By Arzela-Ascoli, a subsequence, still denoted by $u_k$, uniformly converges on $B_3(0)$.  By the closedness of viscosity solutions (cf.\cite{CC}), the subsequence $u_k$ converges uniformly to a continuous viscosity solution, abusing notation, still denoted by $u$, of \eqref{s2} on $B_3(0)$; we included the non-uniformly elliptic convergence proof in the appendix, Lemma \ref{lem:conv}.  By Alexandrov Proposition \ref{prop:Alex}, we deduce that $u$ is second order differentiable almost everywhere on $B_3(0)$.  We fix such a point $x=y$ inside $B_{1/3}(0)$, and let $Q(x)$ be such that $u-Q=o(|x-y|^2)$.

\medskip
Step 2.  We apply Savin's small perturbation theorem \cite{S} to $v_k=u_k-Q$.  Given small $0<r<4/3$, we rescale near $y$:
$$
\bar v_k(\bar x)=\frac{1}{r^2}v_k(r\bar x+y).
$$
Then
\begin{align*}
\|\bar v_k\|_{L^\infty(B_1(0))}&\le \frac{\|u_k(r\bar x+y)-u(r \bar x+y) \|_{L^\infty(B_1(0))}} {r^2} + \frac{\|u(r\bar x+y)-Q(r \bar x+y) \|_{L^\infty(B_1(0))}}{r^2} \\
 &\le \frac{\|u_k(r\bar x+y)-u(r \bar x+y) \|_{L^\infty(B_1(0))}} {r^2} + \sigma(r)
\end{align*}
for some modulus $\sigma(r)=o(r^2)/r^2$. And also $\bar v_k$ solves the elliptic PDE in $B_1(0)$
$$
G(D^2\bar w)=\Delta \bar w+\Delta Q-\sqrt{2+|D^2\bar w+D^2Q|^2}=0.
$$
Note that $\sigma_2(D^2Q)=1$ with $\Delta Q>0$, so $G(0)=0$ with $G(M)$ smooth.  Moreover, $|D^2G|\le C(n)$, and $G(M)$ is uniformly elliptic for $|M|\le 1$, with elliptic constants depending on $n,Q$.  

Now we fix  $r=r(n,Q,\sigma)=:\rho$ small enough such that $\sigma(\rho)<c_1/2,$ where $c_1$ is the small constant in \cite[Theorem 1.3]{S}. As $u_k$ uniformly converges to  $v$, we have $\|\bar v_k\|_{L^\infty(B_1(0))} \le c_1$ for all large enough  $k$.
It follows from \cite[Theorem 1.3]{S} that
$$
\|u_k-Q\|_{C^{2,\alpha}(B_{\rho/2}(y))}\le C(n,Q,\sigma),
$$
with $\alpha=\alpha(n,Q,\sigma)\in(0,1)$. This implies $\Delta u_k\le C(n,Q,\sigma)$ on $B_{\rho/2}(y)$, uniform in $k$.


\medskip
Step 3.   Finally we apply doubling inequality \eqref{doubley} in Corollary \ref{cor:doub} to $u_k$ with $r=\rho/2$:
$$
\sup_{B_{2}(0)}\Delta u_k\le C(n,\rho/2,\|u_k\|_{C^1(B_3(0))})C(n,Q,\sigma)\le C(n,Q,\sigma, A).
$$
We deduce a contradiction to the ``otherwise blowup assumption" at $x=0$.

\begin{rem}
In fact, a similar proof directly establishes interior regularity for viscosity solution $u$ of \eqref{s2} in four dimensions, and then the Hessian estimate, instead of first obtaining the Hessian estimate, then the interior regularity as indicated in the introduction.  By rescaling $\bar u(\bar x)=u(r\bar x+x_0)/r^2$ at various centers, it suffices to show smoothness in $B_1(0)$, if $u\in C(B_5(0))$.  By Alexandrov Proposition \ref{prop:Alex}, we let $x=y$ be a second order differentiable point of $u$ in $B_{1/3}(0)$, with quadratic approximation $Q(x)$ and error $\sigma$ at $y$.  By Savin's small perturbation theorem \cite[Theorem 1.3]{S}, we find a ball $B_\rho(y)$ with $\rho=\rho(n,Q,\sigma)$ on which $u$ is smooth, with estimates depending on $n,Q,\sigma$.  Using \cite{CNS}, we find smooth approximations $u_k\to u$ uniformly on $B_{4}(0)$, with $|Du_k(x)|\le C(\|u\|_{L^\infty(B_4(0))})$ in $B_3(0)$ by the gradient estimate in \cite{T2} and also \cite{CW}.  By the small perturbation theorem \cite[Theorem 1.3]{S}, it follows that $u_k\to u$ in $C^{2,\alpha}$ on $B_{\rho/2}(y)$.  Applying doubling \eqref{doubley} to $u_k$ with $r=\rho/2$, we find that $\Delta u_k\le C(n,Q,\sigma,\|u\|_{L^\infty(B_4(0))})$ on $B_{2}(0)$.  By Evans-Krylov, $u_k\to u$ in $C^{2,\alpha}(B_1(0))$.  It follows that $u$ is smooth on $B_1(0)$.

\smallskip
From interior regularity, a compactness proof for a Hessian estimate would then follow by an application of the small perturbation theorem.  Suppose $u_k\to u$ uniformly but $|D^2u_k(0)|\to\infty$.  We observe that the limit $u$ is interior smooth.  Applying Savin's small perturbation theorem to $u_k-u$, which solves a fully nonlinear elliptic PDE with smooth coefficients, implies a uniform bound on $D^2u_k(0)$ for large $k$, a contradiction.
\end{rem}

\begin{rem}
By combining Alexandrov Proposition \ref{prop:Alex} with [S, Theorem 1.3] as above, we find that general viscosity solutions of $\sigma_2=1$ on $B_1(0)\subset\re^n$ with $\Delta u>0$ have partial regularity: the singular set is closed with Lebesgue measure zero. The same partial regularity also holds for (k-convex) viscosity solutions of equation $\sigma_k=1$, because Alexandrov Proposition \ref{prop:Alex} is valid for such solutions as noted in Remark 4.1.  For (convex) viscosity solutions of the Monge-Amp\`ere equation $\sigma_n=1$, Mooney \cite{M1} showed that the Hausdorff dimension of the singular set is less than $n-1$.  For viscosity solutions of uniformly elliptic equations $F(D^2u)=0$, Armstrong, Silvestre, and Smart \cite{ASS} showed that the Hausdorff dimension of the singular set is at most $n-\epsilon$, where $\epsilon$ depends on $n$ and the constants of ellipticity.
\end{rem}

\section{Appendix}

\begin{proof}[Proof of Lemma \ref{lem:o(r^2)}]  

Choose $x\in B_4(0)$ for which conditions \eqref{gradi}, \eqref{Hess1}, and \eqref{Hess2} are valid.  Given $r>0$ small enough for $B_{2r}(x)\subset B_4(0)$, we just assume $x=0$.  Letting $\eta_\ep(y)=\ep^{-n}\eta(y/\ep)$ be the standard mollifier, we set $u^\ep(y)=\eta_\ep\ast u(y)$ for $|y|<r$.  Letting $Q^\ep(y)=u^\ep(0)+y\cdot Du^\ep(0)+y\cdot D^2u(0)\cdot y/2$, we use Taylor's theorem for the linear part:
$$
u^\ep(y)-Q^\ep(y)=\int_0^1(1-t)y\cdot [D^2u^\ep(ty)-D^2u(0)]\cdot y\,dt.
$$
Letting $\vp\in C^2_c(B_r(0))$ with $|\vp(y)|\le 1$, we average over $B_r=B_r(0)$:
\eqal{
\label{avg}
\dashint_{B_r}\vp(y)(u^\ep(y)-Q^\ep(y))dy&=\int_0^1(1-t)\left(\dashint_{B_r}\vp(y)y\cdot[D^2u^\ep(ty)-D^2u(0)]\cdot y\,dy\right)dt\\
&=\int_0^1\frac{1-t}{t^2}\left(\dashint_{B_{rt}}\vp(t^{-1}z)z\cdot[D^2u^\ep(z)-D^2u(0)]\cdot z\,dz\right)dt.
}
The first term converges to the Radon measure representation of the Hessian:
\begin{align*}
    g^\ep(t)&:=\int_{B_{rt}}\vp(t^{-1}z)z\cdot D^2u^\ep(z)\cdot z\,dz\\
    &\to\int_{B_{rt}}u(z)\pd_{ij}(z^iz^j\vp(t^{-1}z)dz\qquad\text{as }\ep\to0\\
    &=\int_{B_{rt}}\vp(t^{-1}z)z^iz^jd\mu^{ij}\\
    &=\int_{B_{rt}}\vp(t^{-1}z)z\cdot D^2u(z)\cdot z\,dz+\int_{B_{rt}}\vp(t^{-1}z)z^iz^jd\mu^{ij}_s.
\end{align*}
It also has a bound which is uniform in $\ep$:
\begin{align*}
    \frac{g^\ep(t)}{r^nt^{n+2}}&\le\frac{r^2}{(rt)^n}\int_{B_{rt}}|D^2u^\ep(z)|dz\\
    &=\frac{r^2}{(rt)^n}\int_{B_{rt}}\left|\int_{\re^n}D^2\eta_\ep(z-\zeta)u(\zeta)\right|dz\\
    &=\frac{r^2}{(rt)^n}\int_{B_{rt}}\left|\int_{\re^n}\eta_\ep(z-\zeta)d[D^2u](\zeta)\right|dz\\
    &\le \frac{Cr^2}{\ep^n(rt)^n}\int_{B_{rt+\ep}}|B_{rt}(0)\cap B_\ep(\zeta)|\,d\|D^2u\|(\zeta)\\
    &\le \frac{Cr^2}{\ep^n(rt)^n}\min(rt,\ep)^n\|D^2u\|(B_{rt+\ep})\\
    &\le Cr^2\frac{\|D^2u\|(B_{rt+\ep})}{(rt+\ep)^n}\\
    &\le Cr^2.
\end{align*}
In the last inequality, we used \eqref{Hess1} and \eqref{Hess2}, and denoted by $\|D^2u\|$ the total variation measure of $[D^2u]$.  Note also, by \eqref{gradi},
\begin{align*}
    |Du^\ep(0)-Du(0)|&\le \int_{B_\ep}\eta_\ep(z)|Du(z)-Du(0)|dz\\
    &\le C\dashint_{B_\ep}|Du(z)-Du(0)|dz\\
    &=o(1)_\ep.
\end{align*}
By the dominated convergence theorem, we send $\ep\to 0$ in \eqref{avg}:
\begin{align*}
    \dashint_{B_r}\vp(y)(u(y)-Q(y))dy&\le Cr^2\int_0^1\dashint_{B_{rt}}|D^2u(z)-D^2u(0)|dzdt+Cr^2\int_0^1\frac{\|[D^2u]_s\|(B_{rt})}{(rt)^n}dt\\
    &=o(r^2),
\end{align*}
using \eqref{Hess1} and \eqref{Hess2}.  Taking the supremum over all such $|\vp(y)|\le 1$, we conclude $\dashint_{B_r}|h(y)|dy=o(r^2)$.  This completes the proof.
\end{proof}

\begin{proof}[Proof of Lemma \ref{lem:pto(r^2)}]
Given $x\in B_4(0)$ such that \eqref{L^1} and \eqref{point} are true, we let $0<2r<4-|x|$ and $0<\ep<1/2$.  Then by \eqref{L^1},
\begin{align*}
    \left|\{z\in B_r(x):|h(z)|\ge \ep r^2\}\right|&\le \frac{1}{\ep r^2}\int_{B_r(x)}|h(z)|dz\\
    &=\ep^{-1}o(r^n)\\
    &< \ep|B_r(x)|,
\end{align*}
provided $r<r_0(\ep,n,h)$.  Then for each $y\in B_{r/2}(x)$, there exists $z\in B_r(x)$ such that
$$
|h(z)|\le \ep r^2\qquad\text{ and }\qquad |y-z|\le \ep r.
$$
By \eqref{point} and \eqref{L^1}, we obtain for such $y$,
\begin{align*}
    |h(y)|&\le |h(z)|+\frac{|h(y)-h(z)|}{|y-z|}\ep r\\
    &\le \ep r^2+C(n)\ep\dashint_{B_{2r}(x)}|h(\zeta)|d\zeta+C(n,h)\ep r^2\\
    &\le C(n,h)\ep r^2.
\end{align*}
We conclude $\sup_{B_{r/2}(x)}|h(y)|=o(r^2)$. 
\end{proof}

The following is standard, but for lack of reference, we include a proof.
\begin{lem}
\label{lem:conv}
If $u_k\to u$ is a uniformly convergent sequence of viscosity solutions on $B_1(0)$ of a fully nonlinear elliptic equation $F(D^2u,Du,u,x)=0$ continuous in all variables, then $u$ is a viscosity solution of $F$ on $B_1(0)$.
\end{lem}

\begin{proof}
We show it is a subsolution.  Suppose for some $x_0\in B_1(0)$, $0<r<\text{dist}(x_0,\pd B_1(0))$, and smooth $Q$ that $Q\ge u$ on $B_r(x_0)$ with equality at $x_0$.  Set 
$$
Q_\ep=Q+\ep |x-x_0|^2-\ep^4.
$$
We observe that 
$$
u_k(x_0)-Q_\ep(x_0)\ge u(x_0)-Q(x_0)+\ep^4-o(1)_k>0
$$
for $k=k(\ep)$ large enough.  In the ring $B_{\ep}(x_0)\setminus B_{\ep/2}(x_0)$, we have 
$$
u_k(x)-Q_\ep(x)<u(x)-Q(x)-\ep^3/4+\ep^4+o(1)_k<0
$$
for $\ep=\ep(r)$ small enough, and $k=k(\ep)$ large enough.  This means the maximum of $u_k-Q_\ep$ occurs at some in $x_\ep\in B_{\ep/2}(x_0)$.  Since $u_k$ is a subsolution, we get
$$
0\le F(D^2Q_\ep(x_\ep),DQ_\ep(x_\ep),Q_\ep(x_\ep),x_\ep)\to F(D^2Q(x_0),DQ(x_0),Q(x_0),x_0),
$$
as $\ep\to 0$.  This completes the proof.
\end{proof}

\noindent
\bigskip
\textbf{Acknowledgments.}  Y.Y. is partially supported by an NSF grant.

\smallskip
\noindent
DEPARTMENT OF MATHEMATICS, PRINCETON UNIVERSITY, PRINCETON, NJ 08544-1000

\textit{Email address:} rs1838@princeton.edu

\smallskip
\noindent
DEPARTMENT OF MATHEMATICS, UNIVERSITY OF WASHINGTON, BOX 354350, SEATTLE, WA 98195

\textit{Email address:} yuan@math.washington.edu

\end{document}